\documentclass{article}
\usepackage{amsthm,amsmath,amsfonts,latexsym,amssymb,mathrsfs,color,changepage,graphicx}

\usepackage{amsfonts}
\usepackage{amssymb}
\usepackage{graphicx}
\usepackage{pstricks}
\usepackage{amsmath}
\usepackage{amsxtra}
\usepackage{bm}
\usepackage{enumitem}
\usepackage{marginnote}
\usepackage[normalem]{ulem}

\makeatletter
\newcommand* \bigcdot{\mathpalette \bigcdot@{.5}}
\newcommand* \bigcdot@[2]{\mathbin{\vcenter{\hbox{\scalebox{#2}{$\m@th#1\bullet$}}}}}
\makeatother

\makeatletter
\newcommand\appendix@section[1]{%
 \refstepcounter{section}%
 \orig@section*{Appendix \@Alph\c@section: #1}%
 \addcontentsline{toc}{section}{Appendix \@Alph\c@section: #1}%
}
\let\orig@section\section
\g@addto@macro\appendix{\let\section\appendix@section} \makeatother

\theoremstyle{definition}
\newtheorem{Def}{Definition}[section]
\newtheorem{df}[Def]{Definition}
\newtheorem{thm}[Def]{Theorem}
\newtheorem{cor}[Def]{Corollary}
\newtheorem{lem}[Def]{Lemma}

\newtheorem{ex}[Def]{Example}
\newtheorem{q}[Def]{Question}

\renewcommand{\thefootnote}{\fnsymbol{footnote}}

\usepackage{multicol}

\newrgbcolor{lightorange}{1 .5 0}
\newrgbcolor{lightgreen}{0.8 1 .8}
\newrgbcolor{lightyellow}{1 1 0}
\newrgbcolor{lightblue}{0 1 1}

\definecolor{DarkBlue}{rgb}{0,0.2,0.6}
\definecolor{PinkPurple}{rgb}{0.8,0.3,0.3}
\definecolor{darkgreen}{rgb}{.1,.5,0}
\definecolor{brown}{rgb}{.4,.2,.1}

\begin{document}

\title
{Operator-valued rational functions}

\author{Ra{\'u}l\ E.\ Curto, In Sung Hwang and Woo Young Lee}

\date{}

\maketitle

%
%
%
%

\noindent{\bf Abstract.} In this paper we show that every inner
divisor of the operator-valued coordinate function, $zI_E$, is a Blaschke-Potapov
factor. We also introduce a notion of operator-valued ``rational"
function and then show that $\Delta$ is two-sided inner and
rational if and only if it can be represented as a finite
Blaschke-Potapov product; this extends to operator-valued functions the well-known result proved by V.P. Potapov for
matrix-valued functions.


\bigskip

\setcounter{page}{1}


\renewcommand{\thefootnote}{}
\footnote{
\\
\smallskip
}

%
%
%
%

\section{Introduction}

Many properties of matrix-valued functions may not be transferred to operator-valued functions, since some properties of finite matrices
are destined to fail for infinite matrices. \ For example, if $A$ is an
$n\times n$ matrix of $H^2$-functions and $B$ is an $n\times n$
diagonal constant inner function, i.e.,
$B=\hbox{diag}\,(\theta,\cdots,\theta)$ (for $\theta$ an inner
function), then $A$ and $B$ are left coprime if and only if $A$ and
$B$ are right coprime; in other words, left-coprimeness and
right-coprimeness coincide for $A$ and $B$ (cf. \cite[Lemma
C.14]{CHL3}). \ However this is no longer true for operator-valued
functions. \ For example, if $A(z):=S$ (the shift on $\ell^2$)
and $B(z):=\theta I$ (where $I$ is the identity operator on $\ell^2$), then
$A$ and $B$ are right coprime, but not left coprime (cf.
\cite[Example C.12]{CHL3}).

In this paper we consider a question on left inner divisors of the
``operator-valued coordinate" function $zI_E$ (where $E$ is a Hilbert space). \ We consider the
well-known result, proved by V.P. Potapov \cite{Po}, that every rational inner $n\times n$ matrix-valued
function can be written as a finite Blaschke-Potapov product. \ We extend this result to the case
of operator-valued functions.

Let $X$ be a complex Banach space and $\mathbb T$ denote the unit
circle in the complex plane $\mathbb C$. \ For $1\leq p<\infty$, let
$L^p(\mathbb T, X)$ be the linear space of all (equivalence classes
of) strongly measurable functions $f:\mathbb T \rightarrow X$ for
which
$$
\int_{\mathbb T}||f||^pdm <\infty,
$$
where $m$ is the normalized Lebesgue measure on $\mathbb T$. \ We
define $L^{\infty}(\mathbb T, X)$ as the linear space of all
(equivalence classes of) strongly measurable functions $f:\mathbb T
\rightarrow X$ for which there exists $r>0$ such that
$m\bigl(\bigl\{z \in \mathbb T:||f(z)||>r\bigr\}\bigr)=0$. \ Endowed
with the norms
$$
||f||_{L^p(\mathbb T, X)}:=\left(\int_{\mathbb T}||f||^p
dm\right)^{\frac{1}{p}}
$$
and
$$
||f||_{L^{\infty}(\mathbb T, X)}:=\inf  \bigl\{ r>0 : m(\{z\in
\mathbb T:||f(z)||>r\})=0 \bigr\},
$$
the spaces $L^p(\mathbb T, X)$ are complex Banach spaces ($1 \le p \le \infty$). \ For $f \in
L^1(\mathbb T, X)$, the $n$-th Fourier coefficient of $f$, denoted
by $\widehat{f} (n)$, is defined by
$$
\widehat{f}(n):=\int_{\mathbb T}\overline z^nf(z)\, dm(z)\quad
\hbox{for each $n\in\mathbb Z$},
$$
where the integral is understood in the sense of the Bochner
integral. \ For $1\le p \leq\infty$, define
$$
H^p(\mathbb T, X):=\bigl\{f\in L^p(\mathbb T, X): \widehat f(n)=0\ \
\hbox{for}\ n<0 \bigr\}.
$$
Let $\mathcal B(D,E)$ denote the set of all bounded linear operators
between complex Hilbert spaces $D$ and $E$ and abbreviate $B(E,E)$
to $B(E)$. \ For $1\le p\leq\infty$, let $L^p_s(\mathbb T, \mathcal
B(D,E))$ be the set of all (SOT-measurable) $\mathcal B(D,E)$-valued
functions $\Phi$ on $\mathbb T$ such that $\Phi(\cdot)x\in
L^p(\mathbb T, E)$ for each $x\in D$. \ A function $\Phi\in
L^p_s(\mathbb T, \mathcal B(D,E))$ is called a {\it strong
$L^p$-function}. \ We can see that every function in $L^p(\mathbb T,
B(D,E))$ is a strong $L^p$-function. \ The notion of strong $L^2$-function was introduced by V. Peller in \cite{Pe}; the formal theory of strong
$L^2$-functions was developed in \cite{CHL3}.

If $\Phi\in L^1_s (\mathbb T, \mathcal B(D,E))$ and $x \in D$, then
$\Phi(\cdot)x\in L^1_E$. \ The $n$-th Fourier coefficient of $\Phi\in
L^1_s (\mathbb T, \mathcal B(D,E))$, denoted by $\widehat \Phi (n)$,
is given by
$$
\widehat\Phi(n)x:=\widehat{\Phi(\cdot)x}(n) \quad (n\in\mathbb Z, \
x\in D).
$$
We define
$$
H^p_s(\mathbb T, \mathcal B(D,E)) :=\bigl\{\Phi\in L^p_s(\mathbb T,
\mathcal B(D,E)) : \ \widehat\Phi(n)=0\ \hbox{for} \ n<0\bigr\}.
$$
Let $L^{\infty}(\mathbb T, \mathcal B(D,E))$ be the space of all
bounded SOT-measurable $\mathcal B(D,E)$-valued functions on
$\mathbb T$. \ Then we can easily see that
\begin{equation}\label{poiu}
L^{\infty}(\mathbb T, \mathcal B(D,E))=L^{\infty}_s(\mathbb T,
\mathcal B(D,E)).
\end{equation}
Indeed, obviously $L^{\infty}(\mathbb T, \mathcal B(D,E)) \subseteq
L^{\infty}_s(\mathbb T, \mathcal B(D,E))$. \ For the reverse
inclusion, suppose $\Phi\in L^{\infty}_s(\mathbb T, \mathcal
B(D,E))$. \ Then $\Lambda: x \mapsto \Phi(\cdot)x$ is a closed linear
transform from $X$ into $L^{\infty}(\mathbb T, E)$, so that, by the
closed graph theorem, it is bounded. \ Thus for almost all $z \in
\mathbb T$,
$$
||\Phi(z)||_{\mathcal B(D,E))}=\sup_{||x||= 1}||\Phi(z)x||_{E}\leq
||\Lambda||,
$$
which implies $\Phi \in L^{\infty}(\mathbb T, \mathcal B(D,E))$. \ This proves (\ref{poiu}). \ We
will also write $H^{\infty}(\mathbb T, \mathcal B(D,E))\equiv
H^{\infty}_s(\mathbb T, \mathcal B(D,E))$.

\medskip

Write $I_E$ for the identity operator acting on $E$. \ Write
$M_{m\times n}$ for the set of $m\times n$ complex matrices and
abbreviate $M_{n\times n}$ to $M_n$. \ Also we abbreviate $I_{M_n}$ to
$I_n$. \ We say that a function $\Delta \in H^{\infty}(\mathbb T,
\mathcal B(D, E))$ is an {\it inner} function  if
$$
\Delta^*\Delta= I_D\ \ \hbox{a.e. on}\ \mathbb T
$$
and that $\Delta$ is a {\it two-sided inner} function  if
$\Delta\Delta^*=I_E$ a.e. on $\mathbb T$ and $\Delta^*\Delta=I_D$
a.e. on $\mathbb T$. \ For a function $\Phi\in H^\infty(\mathbb
T,\mathcal B(D,E))$, we say that an inner function $\Delta\in
H^\infty(\mathbb T, \mathcal B(D^{\prime},E))$ is a {\it left inner
divisor} of $\Phi$ if $\Phi=\Delta A$ for some $A\in
H^\infty(\mathbb T, \mathcal B(D,D^\prime))$ and that $\Omega \in
H^\infty(\mathbb T, \mathcal B(D,E^{\prime}))$ is an {\it right
inner divisor} of $\Phi$ if $\Phi=B \Omega$ for some $B\in
H^\infty(\mathbb T, \mathcal B(E^\prime, E))$. \ A function $\Delta$
is an {\it inner divisor} of $\Phi$ if it is both a left and a right
inner divisor of $\Phi$. \ As customarily done, we say that two inner
functions $A,B\in H^\infty(\mathbb T, B(E))$ are equal if they are
equal up to a unitary constant right factor, i.e., there exists a
unitary (constant) operator $V\in \mathcal B(E)$ such that $A=BV$.

\medskip

\medskip

Note that if $V$ is a unitary operator in $B(E)$, then for every
$\Phi \in H^{\infty}(\mathbb T, \mathcal B(E))$,
$$
\Phi=V(V^*\Phi)=(\Phi V^*)V,
$$
which implies that $V$ is an inner divisor of $\Phi$.

\medskip

For a function $\Phi\in H^\infty(\mathbb T,B(E))$, we say that a
function $\Delta\in H^\infty(\mathbb T, B(E))$ is a {\it nontrivial
left} (resp. {\it right}) {\it inner divisor} of $\Phi$ if $\Delta$
is a non-unitary operator and is a left (right, resp.) inner divisor
of $\Phi$.

\medskip

For $\alpha\in\mathbb D$, write
$$
b_\alpha(z)\label{blambdaz}:=\frac{z-\alpha}{1-\overline \alpha z},
$$
which is called a {\it Blaschke factor}. \  If $M$ is a closed
subspace of a Hilbert space $E$, then a function of the form
$$
b_\alpha P_M+(I_E-P_M)
$$
is called a  ({\it operator-valued}) {\it Blaschke-Potapov factor},
where $P_M$ is the orthogonal projection of $E$ onto $M$. \ A function
$D$ is called  a ({\it operator-valued}) {\it finite
Blaschke-Potapov product} if $D$ is of the form
$$
D =V \prod_{m=1}^M \Bigl(b_m P_m + (I-P_m)\Bigr),
$$
where $V$ is a unitary operator, $b_m$ is a Blaschke factor, and
$P_m$ is an orthogonal projection in $E$ for each $m=1,\cdots, M$.
In particular, a scalar-valued function $D$ reduces to a finite
Blaschke product $D=\nu \prod_{m=1}^M b_m$, where $\nu=e^{i\omega}$.
\  It is known (cf. \cite{Po}) that an $n\times n$ matrix function
$D$ is rational and inner if and only if it can be represented as a
finite Blaschke-Potapov product. \

On the other hand, we may ask a question: What is a left inner
divisor of $zI_n$\,? For this question, we may guess that each left
inner divisor of $zI_n$ is a Blaschke-Potapov factor. \ More
specifically, we wonder if a left inner divisor of
$\left[\begin{smallmatrix} z&0\\0&z\end{smallmatrix}\right]\equiv
zI_2$ should be of the following form up to a unitary constant right
factor (also up to unitary equivalence):
$$
\begin{bmatrix} 1&0\\ 0&1\end{bmatrix}, \quad \begin{bmatrix} 1&0\\ 0&z\end{bmatrix}, \quad \begin{bmatrix} z&0\\ 0&1\end{bmatrix}\quad\hbox{or}\quad
\begin{bmatrix} z&0\\ 0&z\end{bmatrix}.
$$
For example, $A\equiv \frac{1}{\sqrt{2}}\left[\begin{smallmatrix} 1&-z\\
1&z\end{smallmatrix}\right]$ is a left inner divisor of
$\left[\begin{smallmatrix} z&0\\0&z\end{smallmatrix}\right]\equiv
zI_2$: indeed,
$$
A\cdot
\frac{1}{\sqrt{2}}\begin{bmatrix}
z&z\\-1&1\end{bmatrix}=\begin{bmatrix}z&0\\ 0&z\end{bmatrix}.
$$
In this case, if we take a unitary operator $V:=
\frac{1}{\sqrt{2}}\left[\begin{smallmatrix}
1&1\\-1&1\end{smallmatrix}\right]$, then
$$
\begin{bmatrix} 1&0\\ 0&z\end{bmatrix}
=V\cdot \frac{1}{\sqrt{2}}\begin{bmatrix} 1&-z\\1&z\end{bmatrix} =
\left[V\cdot \frac{1}{\sqrt{2}}\begin{bmatrix}
1&-z\\1&z\end{bmatrix}\cdot V^*\right]\cdot V.
$$
In fact, it was shown in \cite[Lemma 2.5]{CHL1} that
\begin{equation}\label{mq}
\hbox{every left inner divisor of $zI_{n} \in H^\infty(\mathbb T,
M_n)$ is a Blaschke-Potapov factor.}
\end{equation}
This fact is useful for the study of coprime-ness of functions (cf.
\cite{CHL2}). \ In \cite[p.23]{CHL3}, the authors asked:
\smallskip

\begin{q}\label{Q1.1} Is the statement in (\ref{mq}) still true for
operator-valued functions\,?
\end{q}
\smallskip

We will call a function of the form $z I_E$ the operator-valued coordinate function. \ This allows us to rephrase Question
\ref{Q1.1} as follows: Is every left inner divisor of the
operator-valued coordinate function a Blaschke-Potapov factor\,?
 In
Section 2, we give an affirmative answer to this question. \ In
Section 3, we introduce a notion of operator-valued ``rational"
function and then show
 that $\Delta$ is two-sided inner and rational if and only if
it can be represented as a finite Blaschke-Potapov product, which
extends the well-known result for the matrix-valued functions due to
V.P. Potapov \cite{Po}. \ In Section 4, we consider coprime
operator-valued rational functions. \ Lastly, in Section 5, we take a
glance at right coprime-ness and subnormality of Toeplitz operators,
aiming at shedding new light on the differences between matrix-valued
functions and operator-valued functions.

\medskip

To proceed, we give an elementary observation.

 If $\dim E<\infty$ and  $\Theta \in
H^{\infty}(\mathbb T, \mathcal B(E))$ is a two-sided inner function,
then any left inner divisor of $\Theta$ is two-sided inner (cf.
\cite[Lemma 4.10]{CHL2}). \ We can say more:

\medskip

\begin{lem}\label{sxfveryhnt} \ (cf. \cite[Lemma 2.2]{CHL3}) \ Let $\Theta\in H^{\infty}(\mathbb T, \mathcal B(E))$
 be a two-sided inner function. \ If $\Delta
\in H^{\infty}(\mathbb T, \mathcal B(E))$ is a left inner divisor of
$\Theta$, then $\Delta$ is two-sided inner.
\end{lem}
\begin{proof} If $\Delta
$ is a left inner divisor of $\Theta$, we may write
$$
\Theta=\Delta \Omega \quad \hbox{for some} \ \Omega \in
H^{\infty}(\mathbb T, \mathcal B(E)).
$$
Thus $\Delta$ is surjective, and hence unitary a.e. on $\mathbb T$.
Thus $\Delta$ is two-sided inner.
\end{proof}

In particular, if $\Theta:=\theta I_E$ (for $\theta$ a scalar inner
function) then every left inner divisor of $\Theta$ is an inner
divisor of $\Theta$. \ However, in general, a left inner divisor of a
two-sided inner function need not be its right inner divisor. \ To see
this, we first observe:

\medskip

\begin{lem}\label{dcvgbhnlkjn} \ Let $\Theta\in H^{\infty}(\mathbb T, \mathcal B(E))$
 be a two-sided inner function and $\Delta
\in H^{\infty}(\mathbb T, \mathcal B(E))$ be a left inner divisor of
$\Theta$. \  Then $\Delta$ is an inner divisor of $\Theta$ if and only
if $\Theta\Delta^*\in H^{\infty}(\mathbb T, \mathcal B(E))$.
\end{lem}
\begin{proof} Let $\Theta\in H^{\infty}(\mathbb T, \mathcal B(E))$
 be a two-sided inner function and $\Delta
\in H^{\infty}(\mathbb T, \mathcal B(E))$ be a left inner divisor of
$\Theta$. \ Then, by Lemma \ref{sxfveryhnt}, $\Delta$ is two-sided
inner and we may write
$$
\Theta=\Delta\Omega \quad(\Omega \in H^{\infty}(\mathbb T, \mathcal
B(E))).
$$
Suppose that $\Delta$ is an inner divisor of $\Theta$. \ Then we can also
write $\Theta=\Psi \Delta$ for some $\Psi \in H^{\infty}(\mathbb T, \mathcal
B(E)))$. \ Thus we have that
$\Theta\Delta^*=\Psi \Delta \Delta^*=\Psi \in H^{\infty}(\mathbb T,
\mathcal B(E))$.
The converse is obvious.
\end{proof}

\medskip

We then have:

\begin{ex} Let $\{e_n: n \in \mathbb Z\}$ be the canonical orthonormal basis for
$L^2(\mathbb T)$. \ Define $\Delta$ and $\Theta$ in
$H^{\infty}(\mathbb T, \mathcal B(L^2(\mathbb T)))$ by
$$
\Delta(z)e_n:=\begin{cases} e_{n+1}z \quad \hbox{if} \ n \geq 0\\
e_{n+1} \quad  \ \ \hbox{if} \ n < 0
\end{cases} \qquad \hbox{and} \qquad \Theta(z)e_n:=\begin{cases} e_{-n+1}z^2 \quad \hbox{if} \ n \leq 1\\
e_{-n+1} \quad  \quad   \hbox{if} \ n > 1.
\end{cases}
$$
Then $\Theta$ and $\Delta$ are two-sided inner. \ Observe that
$$
\Delta^*(z)e_n=\begin{cases} e_{n-1}z^{-1} \quad \hbox{if} \ n \geq 1\\
e_{n-1} \quad  \quad \ \ \hbox{if} \ n < 1\end{cases} \ \hbox{and
hence} \quad
\Delta^*(z)\Theta(z)e_n=\begin{cases} e_{-n}z \quad \quad \hbox{if} \ n < 1\\
e_{-1}z^2\quad  \ \ \hbox{if} \ n=1\\ e_{-n} \quad \quad \ \
\hbox{if} \ n>1.
\end{cases}
$$
Thus $\Delta$ is a left inner divisor of $\Theta$. \ Since $
\Theta(z)\Delta^*(z)e_3=e_{-1}z^{-1}$, it follows from Lemma
\ref{dcvgbhnlkjn} that $\Delta$ is not a right inner divisor of
$\Theta$.
\end{ex}

\medskip

On the other hand, Lemma \ref{sxfveryhnt} may fail if ``left" is replaced by
``right."

\medskip

\begin{ex} Let $S$ be the shift operator on $H^2(\mathbb T)$ defined by
$$
(Sf)(z):=zf(z) \quad (f \in H^2(\mathbb T), \ z \in \mathbb T)
$$
and let $\Delta(z):=S \in H^{\infty}(\mathbb T, \mathcal
B(H^2(\mathbb T)))$. \ Then
$$
\Delta(z)^*\Delta(z)= S^*S=I,
$$
which implies that $\Delta$ is a right inner divisor of (a two-sided
inner function) $I$. \ But $\Delta$ is not two-sided inner. \ By Lemma
\ref{sxfveryhnt}, $\Delta$ is not a left inner divisor of $I$.
\end{ex}

\bigskip


\section[Inner divisors of coordinate functions]{Inner divisors of the operator-valued coordinate functions}

\medskip

For a complex Banach space $X$ and an open subset $G$ of $\mathbb
C$, a function $A: G \rightarrow X$ is called holomorphic if, in any
sufficiently small disc $D(\lambda,r)=\{\zeta : |\zeta-\lambda|<r\}
\subset G$, $A$ is the sum of convergent power series
$$
A(\zeta)=\sum_{n=0}^{\infty}\widehat{A}(n)(\zeta-\lambda)^n \quad
 (\widehat{A}(n) \in X).
$$
Denote by $\hbox{Hol}(G, X)$ the space of all holomorphic functions
$A:G \rightarrow X$.
Let us associate to any function $f$ on $\mathbb D$ a family of
function $f_r$ on $\mathbb T$, defined by
$$
f_r(z):=f(rz) \quad (0\leq  r <1).
$$
For $1\leq p <\infty$, let $H^p(\mathbb D, X)$ be the set of all
functions $f \in \hbox{Hol}(\mathbb D, X)$ satisfying
$$
||f||_{H^p(\mathbb D, X)}:=\biggl(\sup_{0<r<1}\int_{\mathbb
T}||f_r||^pdm \biggr)^{\frac{1}{p}}<\infty.
$$
We define  $H^{\infty}(\mathbb D, X)$ be the set of all bounded
functions $f \in \hbox{Hol}(\mathbb D, X)$.

\medskip

If $\Phi \in L^{\infty}(\mathbb T, \mathcal{B}(D,E))$ and $\zeta = r e^{i\theta} \in \mathbb{D}$, let $P[\Phi]$ denote the Poisson integral
of $f$ defined by
\begin{equation} \label{Poisson}
P[\Phi](\zeta)x:=\frac{1}{2\pi}\int_{-\pi}^{\pi}P_r(\theta-t)\Phi(e^{it})x \, dt \quad(x \in D; \, 0\leq r<1).
\end{equation}
It is known (cf. \cite[Lemma 2.1]{GHKL}) that if $\Phi \in
H^{\infty}(\mathbb T, \mathcal B(D, E))$ and $\Psi \in
H^{\infty}(\mathbb T, \mathcal B(D^{\prime}, D))$, then
\begin{equation}\label{lemma4.treghc1000001211}
P[\Phi \Psi]=P[\Phi]P[\Psi].
\end{equation}

\medskip

We now consider Question \ref{Q1.1}. \ We actually wish to study a more general case, and
we first observe that
if $A$ is an inner divisor of $z^NI_E$, then there exists a function
$\Omega \in H^{\infty}(\mathbb T, \mathcal{B}(E))$ such that
$$
A \Omega=\Omega A=z^NI_E \quad \hbox{a.e. on} \ \mathbb T,
$$
so that
$$
A^* z^N I_E=\Omega \in H^{\infty}(\mathbb T, \mathcal{B}(E)),
$$
which implies that $A$ is a polynomial of degree at most $N$.

\medskip

We are ready for:

\begin{thm}\label{asdcvfgb,mnbv} Let $A$ be a polynomial of degree
$N$. \ Then the following are equivalent:
\begin{itemize}
\item[(a)] $A$ is an inner divisor of $z^NI_E$;
\item[(b)] $\widehat{A}(n)$ and $\widehat{A}(n)^*$ are partial isometries for each
$n=0,1,\cdots, N$, and
$$
E=\bigoplus_{n=0}^{N} \hbox{ran}\,
\widehat{A}(n)=\bigoplus_{n=0}^{N} \hbox{ran}\, \widehat{A}(n)^*;
$$
\item[(c)] $A$ is a finite Blaschke-Potapov product of the form
$$
A(z)=V\prod_{m=1}^{N}\Bigl(z P_m + (I-P_m)\Bigr),
$$
where $P_m$ is the orthogonal projection from $E$ onto
$\bigoplus_{n=m}^N\hbox{ran}\, \widehat{A}(n)^*$, and
$V:=\hbox{diag}\bigl(\widehat{A}(0)|_{\hbox{ran}\,
\widehat{A}(0)^*}, \widehat{A}(1)|_{\hbox{ran}\, \widehat{A}(1)^*},
\cdots, \widehat{A}(N)|_{\hbox{ran}\, \widehat{A}(N)^*}\bigr)$.
\end{itemize}
\end{thm}

\begin{proof} (a) $\Rightarrow$ (b):
Suppose that $A$ is an inner divisor of $z^NI_E$. \ Without loss of
generality we  may assume that $A_0$ and $A_N$ are nonzero. \ Write
$$
A^{(1)}(z):=\sum_{n=1}^N \widehat{A}(n) z^{n-1}.
$$
Then
\begin{equation}\label{aSRYJUPKJN}
A(z)=\widehat{A}(0)+zA^{(1)}(z).
\end{equation}
By Lemma \ref{sxfveryhnt}, $A$ is two-sided inner. \ Thus, for almost
all $z \in \mathbb T$,
\begin{equation}\label{vcxzsdfdxz}
A^{(1)}(z)^*\widehat{A}(0)=0 \quad \hbox{and} \quad
\widehat{A}(0)\widehat{A}(0)^*+A^{(1)}(z)A^{(1)}(z)^*=I_E,
\end{equation}
so that
$$
\widehat{A}(0)\widehat{A}(0)^*\widehat{A}(0)=\widehat{A}(0)\quad\hbox{and}\quad
A^{(1)}(z)A^{(1)}(z)^* A^{(1)}(z)=A^{(1)}(z),
$$
which implies that $\widehat{A}(0)$ and $A^{(1)}(z)$
are
partial isometries for almost all $z \in \mathbb T$ (cf. \cite{Ha}).
Since $A$ is
inner, it follows from (\ref{aSRYJUPKJN}) that, for almost all $z \in
\mathbb T$,
$$
\ker \widehat{A}(0) \bigcap \ker A^{(1)}(z)=\{0\}.
$$
Since $(H\cap K)^\perp=H^\perp \bigvee K^\perp$ for closed subspaces
$H,K$ of $E$, it follows that
\begin{equation}\label{wsdsxfcvbnlmio,;l}
\hbox{ran}\, \widehat{A}(0)^* \bigvee \hbox{ran}\, A^{(1)}(z)^*=E.
\end{equation}
On the other hand, it follows from (\ref{aSRYJUPKJN}) that
$\hbox{ran}\, \widehat{A}(0)^* \subseteq \bigl(\hbox{cl ran}\,
A^{(1)}(z)^*\bigr)^{\perp}$ and, by (\ref{wsdsxfcvbnlmio,;l}), we
have that
$$
\hbox{ran}\, \widehat{A}(0)^* \bigoplus^\perp \hbox{ran}\, A^{(1)}(z)^*=E.
$$
Thus, for almost all $z \in \mathbb T$,
$$
\bigl(\ker A^{(1)}(z)\bigr)^{\perp}=\hbox{ran}\, A^{(1)}(z)^*=\ker
\widehat{A}(0).
$$
Similarly, we can show that, for almost all $z \in \mathbb T$,
\begin{equation}\label{dfvgbhnjhbvcxz}
\hbox{ran}\, A^{(1)}(z)=\ker \widehat{A}(0)^*=\bigl(\hbox{ran}\,
\widehat{A}(0)\bigr)^{\perp}.
\end{equation}
Since $A^{(1)}(z)$ is a partial isometry, $A^{(1)}|_{\ker
\widehat{A}(0)}: \mathbb T \to \mathcal B(\ker \widehat{A}(0),
(\hbox{ran}\, \widehat{A}(0))^{\perp})$ is two-sided inner. \ Now we
will show that
\begin{equation}\label{sdfghblkj}
\hbox{ran}\, \widehat{A}(1) \subseteq \bigl(\hbox{ran}\,
\widehat{A}(0)\bigr)^{\perp}.
\end{equation}
For  $x \in E$, it follows from (\ref{dfvgbhnjhbvcxz}) that
$$
\aligned A^{(1)}(z)x&=\sum_{n=1}^N \widehat{A}(n)x z^{n-1} \in
\bigl(\hbox{ran}\, \widehat{A}(0)\bigr)^{\perp} \ \hbox{for almost all} \ z \in \mathbb T \\
&\Longrightarrow \bigl\langle \widehat{A}(n) x, ~ \widehat{A}(0)y \bigr \rangle=0 \quad \hbox{for all} \ y \in E\\
&\Longrightarrow P[A^{(1)}](\zeta)x =\sum_{n=1}^N \widehat{A}(n)x
\zeta^{n-1} \in \bigl(\hbox{ran}\, \widehat{A}(0)\bigr)^{\perp} \
\hbox{for all} \ \zeta
\in \mathbb D\\
&\Longrightarrow \widehat{A}(1)x=P[A^{(1)}](0)x \in
\bigl(\hbox{ran}\, \widehat{A}(0)\bigr)^{\perp} \, ,
\endaligned
$$
which proves (\ref{sdfghblkj}). \ Put $A^{(2)}(z):=\sum_{n=2}^N
\widehat{A}(n) z^{n-2}$. \ Then, by the same argument, we have that
\begin{itemize}
\item[(i)] $\widehat{A}(1)$ is a partial isometry;
\item[(ii)] $(\ker A^{(2)}(z))^{\perp}=\ker \widehat{A}(1)$ and $\hbox{ran}\, A^{(2)}(z)=\bigl(\hbox{ran}\,
\widehat{A}(0) \bigoplus \hbox{ran}\, \widehat{A}(1)\bigr)^{\perp}$;
\item[(iii)] $A^{(2)}|_{\ker \widehat{A}(1)}: \mathbb T \to \mathcal B(\ker \widehat{A}(1),
(\hbox{ran}\, \widehat{A}(1))^{\perp})$ is two-sided inner;
\item[(iv)] $\hbox{ran}\, \widehat{A}(2) \subseteq \bigl(\hbox{ran}\, \widehat{A}(0) \bigoplus \hbox{ran}\, \widehat{A}(1)\bigr)^{\perp}$.
\end{itemize}
Continuing this process, we have that $\widehat{A}(n)$ is a partial
isometry for each $n=0,1,\cdots, N$, and
$$
\hbox{ran}\, \widehat{A}(N) \subseteq \biggl(\bigoplus_{n=0}^{N-1}
\hbox{ran}\, \widehat{A}(n) \biggr)^{\perp}.
$$
But since $A(z)$ is unitary for almost all $z \in \mathbb T$, we
have that
$$
\bigoplus_{n=0}^{N} \hbox{ran}\, \widehat{A}(n)  = E.
$$
Similarly, we also have that $\widehat{A}(n)^*$ is a partial
isometry for each $n=0,1,\cdots, N$, and
$$
E=\bigoplus_{n=0}^{N} \hbox{ran}\, \widehat{A}(n)^*.
$$

\medskip
(b) $\Rightarrow$ (c): Suppose $\widehat{A}(n)$ and
$\widehat{A}(n)^*$ are partial isometries for each $n=0,1,\cdots,
N$, and
$$
E=\bigoplus_{n=0}^{N} \hbox{ran}\,
\widehat{A}(n)=\bigoplus_{n=0}^{N} \hbox{ran}\, \widehat{A}(n)^*.
$$
Write $E_n:= \hbox{ran}\, \widehat{A}(n)$ and $F_n:=\hbox{ran}\,
\widehat{A}(n)^*$. \ Then we can write
$$
A(z)=\begin{bmatrix}
\widehat{A}(0)|_{F_0}&0&0&\cdots&0\\
0&(\widehat{A}(1)|_{F_1})z&0&\cdots&0\\
\vdots&0&(\widehat{A}(2)|_{F_2})z^2&\ddots&\vdots\\
\vdots&&\ddots&\ddots&0\\
0&0&\cdots&0&(\widehat{A}(N)|_{F_N})z^N
\end{bmatrix} : \bigoplus_{n=0}^{N} F_n \to  \bigoplus_{n=0}^{N}E_n.
$$
Let $P_m$ be the orthogonal projection from $E$ onto
$\bigoplus_{k=m}^N F_k \  (m=1,2,\cdots, N)$. \ Then
$$
A(z)=V\prod_{m=1}^{N}\Bigl(z P_m + (I-P_m)\Bigr),
$$
where $V:=\hbox{diag}(\widehat{A}(0), \widehat{A}(1), \cdots,
\widehat{A}(N))$ is unitary.
\medskip

(c) $\Rightarrow$ (a): Obvious.

\medskip

This completes the proof.
\end{proof}

\medskip

The following corollary gives an affirmative answer to Question
\ref{Q1.1}.

\begin{cor}
If $\Delta\in H^\infty(\mathbb T, \mathcal B(E))$ is a left inner
divisor of $zI_E$, then $\Delta$ is a Blaschke-Potapov factor.
\end{cor}

\begin{proof}
Immediate from Theorem \ref{asdcvfgb,mnbv}.
\end{proof}

\medskip

We recall:

\begin{lem}\label{ythbbgrvfedcw} \cite[Theorem 3.11.10]{Ni}
Let $A \in H^{\infty}(\mathbb D, \mathcal B(D,E))$. \ Then the nontangential SOT limit
$$
\lim_{r \rightarrow 1}A(rz)\equiv bA(z)
$$
exists for almost all $z \in \mathbb T$, such that $bA \in
H^{\infty}(\mathbb T, \mathcal B(D, E))$ and
$$
A(\zeta)x= P[bA](\zeta)x ,
$$
for $x \in D$ and $\zeta\in\mathbb D$. \ The Taylor coefficients of $A$
coincide with the nonnegatively-indexed Fourier coefficients of $b A$.
Moreover, the mapping $b: A \longrightarrow bA$ is an isometric
bijection from $H^{\infty}(\mathbb D, \mathcal B(D, E))$ onto
$H^{\infty}(\mathbb T, \mathcal B(D, E))$.
\end{lem}

\bigskip

For $\Phi \in H^{\infty}(\mathbb T,\mathcal B(E))$ and $\alpha \in
\mathbb D$, write
$$
\Phi_{\alpha}:= \Phi \circ b_{\alpha}.
$$
Then we can easily check the following (cf. \cite{GHKL}):
\begin{itemize}
\item[(a)] $\Phi_{\alpha} \in H^{\infty}(\mathbb T,\mathcal B(E))$;
\item[(b)] If $\Delta$ is an inner function with values in $\mathcal B(E)$,
then so is $\Delta_{\alpha}$.
\end{itemize}

\medskip

We then have:

\medskip

\begin{cor}\label{afrvghnjmnbv} Let $A \in H^{\infty}(\mathbb T,
\mathcal B(E))$. \ Then the following are equivalent:
\begin{itemize}
\item[(a)] $A$ is an inner divisor of $b_{\alpha}^NI_E$;
\item[(b)] $A$ is a polynomial in $b_{\alpha}$ of degree at most $N$, of the form
$$
A=\sum_{n=0}^{N}A_nb_{\alpha}^n,
$$
where $A_n$ and $A_n^*$ are partial isometries for each
$n=0,1,\cdots, N$, and
$$
E=\bigoplus_{n=0}^{N} \hbox{ran}\, A_n=\bigoplus_{n=0}^{N}
\hbox{ran}\, A_n^*;
$$
\item[(c)] $A$ is a finite Blaschke-Potapov product of the form
$$
A=V\prod_{m=1}^{N}\Bigl(b_{\alpha} P_m + (I-P_m)\Bigr),
$$
where $P_m$ is the orthogonal projection from $E$ onto
$\bigoplus_{n=m}^N\hbox{ran}\, A_n^*$, and
$V:=\hbox{diag}(A_0|_{\hbox{ran}\, A_0^*}, A_1|_{\hbox{ran}\,
A_1^*}, \cdots, A_N|_{\hbox{ran}\, A_N^*})$.
\end{itemize}
\end{cor}

\begin{proof} Observe that
\begin{itemize}
\item[(i)] $A(z)=\sum_{n=0}^{N}A_nb_{\alpha}(z)^n
\Longleftrightarrow A_{-\alpha}(z)=\sum_{n=0}^{N}A_nz^n$;
\item[(ii)] $A$ is an inner divisor of $b_{\alpha}^NI_E
\Longleftrightarrow$ $A_{-\alpha}$ is an inner divisor of $z^NI_E$;
\item[(iii)] $A(z)=V\prod_{m=1}^{N}\Bigl(b_{\alpha} P_m + (I-P_m)\Bigr)
\Longleftrightarrow A_{-\alpha}(z)=V\prod_{m=1}^{N}\Bigl(z P_m +
(I-P_m)\Bigr)$.
\end{itemize}
Thus the result follows at once from Theorem \ref{asdcvfgb,mnbv}.
\end{proof}

\medskip

\begin{cor}\label{afrvdfsnbv} Let $A \in H^{\infty}(\mathbb T,
\mathcal B(E))$. \ If $A$ is a nontrivial inner divisor of
$b_{\alpha}^NI_E$, then $\widehat{A_{-\alpha}}(n)\neq 0$ for some
$n=1,2,\cdots, N$.
\end{cor}

\begin{proof}
Immediate from Corollary \ref{afrvghnjmnbv}.
\end{proof}

\medskip

%
%
%
%

\section{Operator-valued rational functions}

\medskip

In this section we will introduce the notion of operator-valued
``rational" function. \ Recall that a matrix-valued function is
rational if its entries are rational functions. \ But this definition
is not appropriate for operator-valued functions, in particular
$H^{\infty}$-functions, even though the terminology of ``entry'' may
be properly interpreted. \ Thus, the new idea should capture and encapsulate a definition of operator-valued rational function which is equivalent to the condition that each entry is
rational when the function is matrix-valued. \ In the sequel, we give a formal definition of operator-valued rational function.

To do so, we first recall Toeplitz and Hankel operators.
For $\Phi\in L^{\infty}(\mathbb T, \mathcal{B}(D,E)) $, the Hankel
operator $H_\Phi: H^2(\mathbb T, D) \to H^2(\mathbb T, E)$ is
defined by
$$
H_\Phi f:=J P_-(\Phi f)  \quad(f \in H^2(\mathbb T, D)),
$$
where $J$ denotes the unitary operator from $L^2(\mathbb T,E)$ to
$L^2(\mathbb T,E)$ given by $(Jg)(z) :=\overline{z} g(\overline{z})$
for $g \in L^2(\mathbb T, E)$ and $P_-$ is the orthogonal
projection from $L^2(\mathbb T, E)$ onto $L^2(\mathbb T, E) \ominus
H^2(\mathbb T, E)$. \ Also a Toeplitz operator $T_{\Phi}: H^2(\mathbb
T, D) \to H^2(\mathbb T, E)$ is defined by
$$
T_{\Phi}f:= P_+(\Phi f) \quad(f \in H^2(\mathbb T, D)),
$$
where $P_+$ is the orthogonal projection from $L^2(\mathbb T, E)$
onto $H^2(\mathbb T, E)$.

\medskip

For a $B(D,E)$-valued function $\Phi$, write
$$
\breve\Phi (z):=\Phi(\overline z)\quad\hbox{and}\quad
\widetilde\Phi:=\breve\Phi^*.
$$

As usual, a {\it shift} operator $S_E$ on $H^2(\mathbb T, E)$ is
defined by
$$
(S_E f)(z):=zf(z)\quad \hbox{for each} \ f\in H^2(\mathbb T, E).
$$

\medskip

The following theorem is a fundamental result in modern operator
theory.

\medskip

\noindent {\bf The Beurling-Lax-Halmos Theorem.}\label{beur}
\cite{CHL3}, \cite{Ha}, \cite{FF}, \cite{Pe} A subspace
$M$ of $H^2(\mathbb T, E)$ is invariant for the shift operator $S_E$
on $H^2(\mathbb T, E)$ if and only if
$$
M=\Delta H^2(\mathbb T, E^{\prime}),
$$
where $E^{\prime}$ is a subspace of $E$ and $\Delta$ is an inner
function with values in $\mathcal B(E^{\prime}, E)$. \  Furthermore,
$\Delta$ is unique up to a unitary constant right factor, i.e., if
$M=\Theta H^2(\mathbb T, E^{\prime \prime})$, where $\Theta$ is an
inner function with values  in $\mathcal B(E^{\prime\prime}, E)$,
then $\Delta=\Theta V$, where $V$ is a unitary operator from
$E^{\prime}$ onto $E^{\prime\prime}$.

\bigskip

If $\Phi\in L^{\infty}(\mathbb T, \mathcal B(D,E))$ then, by the
Beurling-Lax-Halmos Theorem,
$$
\hbox{ker}\, H_{\Phi^*}=\Delta H^2(\mathbb T, E^{\prime})
$$
for some inner function $\Delta$ with values in $\mathcal
B(E^{\prime}, E)$. \ We note that $E^{\prime}$ may be the zero space
and $\Delta$ need not be two-sided inner.

\medskip

We recall:

\begin{df}\label{dfbundd} \cite{CHL3}
A function $\Phi\in L^2_s(\mathbb T, B(D,E))$ is
said to be of {\it bounded type} if $\hbox{ker}\,
H_{{\Phi}}^*=\Theta H^2_{E}$ for some two-sided inner function
$\Theta$ with values in $\mathcal B(E)$.
\end{df}

\medskip

It is known that if $\phi\equiv (\phi_{ij})$ is a matrix-valued
function of bounded type then each entry $\phi_{ij}$ is of bounded
type, i.e., $\phi_{ij}$ is a quotient of two bounded analytic
functions. \ In particular, it is also known (\cite{Ab}, \cite{CHL1})
that if $\phi\in L^2$ is of bounded type then $\phi$ can be written
as
\begin{equation}\label{bts}
\phi=\overline{\theta} a\quad \hbox{(where $\theta$ is an inner
function and $a\in H^2$).}
\end{equation}

\medskip

For $\Phi\in H^{\infty}(\mathbb T, \mathcal B(D_1,E))$ and $\Psi\in
H^{\infty} (\mathbb T, \mathcal B(D_2,E))$, we say that $\Phi$ and
$\Psi$ are {\it left coprime} if $\Phi$ and $\Psi$ has no common
nontrivial left inner divisor. \  Also, we say that $\Phi$ and $\Psi$
are {\it right coprime} if $\widetilde\Phi$ and $\widetilde\Psi$ are
left coprime.

\medskip

\begin{lem}\label{rem2.4} \cite[Lemma 2.4]{CHL3}
If $\Phi\in L^\infty (\mathbb T, \mathcal B(D,E))$ and $\Delta$ is a
two-sided inner function with values in $\mathcal B(E)$, then the
following are equivalent:
\begin{itemize}
\item[(a)] $\breve\Phi$ is of bounded type, i.e., $\ker H_{\Phi^*}=\Delta H^2(\mathbb T, E)$;
\item[(b)] $\Phi=\Delta A^*$,
where $A\in H^\infty (\mathbb T, \mathcal B(E,D))$ is such that
$\Delta$ and $A$ are right coprime.
\end{itemize}
\end{lem}

\medskip

We now introduce:

\begin{df}\label{ddd} A function $\Phi \in H^{\infty}(\mathbb T, \mathcal B(D, E))$ is said to be {\it
rational} if
\begin{equation}\label{def3.3}
\theta H^2(\mathbb T, E) \subseteq
\hbox{ker}\,H_{\Phi^*}
\end{equation}
for some finite Blaschke product $\theta$.
\end{df}

\medskip

Observe that if $\Phi \in H^{\infty}(\mathbb T, \mathcal B(D, E))$,
then
\begin{equation}\label{rsftuy}
\Phi\ \hbox{is rational} \Longrightarrow \breve{\Phi}\ \hbox{is of
bounded type.}
\end{equation}
To see this, suppose $\Phi$ is rational. \ By definition and the
Beurling-Lax-Halmos Theorem there exist a finite Blaschke product
$\theta$ and an inner function $\Delta \in H^{\infty}(\mathbb T,
\mathcal B(E^{\prime}, E)))$ such that
$$
\theta H^2(\mathbb T, E) \subseteq \hbox{ker}\,H_{\Phi^*}=\Delta
H^2(\mathbb T, E^{\prime}),
$$
which implies that $\Delta$ is a left inner divisor of $\theta
I_{E}$. \  Thus $\Delta$ is two-sided inner, so that by Lemma
\ref{rem2.4}, $\breve{\Phi}$ is of bounded type, which proves
(\ref{rsftuy}).

\medskip

Also, if $\Phi\equiv (\phi_{ij})\in H^\infty(\mathbb T, M_{m\times
n})$ is a  rational function in the sense of Definition \ref{ddd},
then each entry $\phi_{ij}$ is rational. \ To see this suppose a
matrix-valued  function $\Phi$ satisfies the condition
(\ref{def3.3}). \ Put $A:=\theta \Phi^*$. \ Then $A\in H^\infty(\mathbb
T, M_{n \times m})$ and $\Phi=\theta A^*$. \  Thus $\phi_{ij}$ can be written
as
$$
\phi_{ij}=\theta\overline{a_{ij}} \quad(a_{ij}\in H^{\infty}).
$$
Via the Kronecker's Lemma \cite[p.183]{Ni1}, we can see that
$$
\phi_{ij}\ \hbox{is rational}\Longleftrightarrow
\phi_{ij}=\theta\overline{a_{ij}}\ \hbox{with a finite Blaschke
product}\ \theta,
$$
which says that  each $\phi_{ij}$ is rational.

In particular, if $\theta=z^n$ in (\ref{def3.3}), $\Phi$ becomes an
operator-valued polynomial.

\medskip

In 1955, V.P. Potapov \cite{Po} proved that  an $n\times n$
matrix-valued function $\Phi$ is rational and inner if and only if
it can be represented as a finite Blaschke-Potapov product. \ In this
section, we extend this result to operator-valued functions. \ To so
so, we first observe:

\medskip

\begin{lem}\label{saxfvedrtyhb} Suppose that $\theta$ is a finite Blaschke product of the form
$$
\theta=\prod_{n=1}^M b_{\alpha_n} \quad(\alpha_n \in \mathbb D)
$$
and $\Delta$ is an inner divisor of $\Theta=\theta I_E$. \ Let
$\Omega:=\Theta \Delta^*$ and $P_n$ be the orthogonal projection of
$E$ onto $\hbox{cl ran}\, P[\Omega](\alpha_n)$. \ Then $\Delta
\bigl(b_{\alpha_n}P_{n}+(I_E-P_{n}) \bigr)^*$ is an inner divisor of
$\Theta_n:=\theta \overline{b_{\alpha_n}} I_E$  for each
$n=1,2,\cdots, M$.
\end{lem}

\begin{proof} Write
$$
A:= P[\Delta]\quad\hbox{and}\quad C:= P[\Omega].
$$
Then it follows from (\ref{lemma4.treghc1000001211}) that
$$
P[\Theta]=AC=CA.
$$
Since $0=P[\Theta](\alpha_n)=A(\alpha_n)C(\alpha_n)$, we have that 
$\hbox{cl ran}\, C(\alpha_n) \subseteq \ker A(\alpha_n)$.
Thus we may write
\begin{equation}\label{qwsrfveyhtlkl}
A(\alpha_n)=A(\alpha_n)\bigl(b_{\alpha_n}P_n+(I_E-P_n)\bigr),
\end{equation}
where $P_n$ is the orthogonal projection of $E$ onto $\hbox{cl
ran}\, C(\alpha_n)$. \ On the other hand, we can write
$$
A- A(\alpha_n)=A_n (b_{\alpha_n} I_E) \quad \hbox{for some} \ A_n
\in H^{\infty}(\mathbb D, \mathcal B(E)).
$$
It thus follows from (\ref{qwsrfveyhtlkl}) that
$$
\aligned A= &A(\alpha_n)\bigl(b_{\alpha_n}P_n
+(I_E-P_{n})\bigr)\\&+A_n\bigl(P_{n}+b_{\alpha_n}(I_E-P_{n})\bigr)
\bigl(b_{\alpha_n}P_{n}+(I_E-P_{n})\bigr)\\
&=\Bigl[A(\alpha_n)+A_n\bigl(P_{n}+b_{\alpha_n}(I_E-P_{n})\bigr)
\Bigr]\bigl(b_{\alpha_n}P_{n}+(I_E-P_{n})\bigr).
\endaligned
$$
Since $A(\alpha_n)+A_n\bigl(P_{n}+b_{\alpha_n}(I_E-P_{n})\bigr)\in
H^{\infty}(\mathbb D, \mathcal B(E))$, it follows from Lemma \ref{ythbbgrvfedcw}  that
$$
\Delta_n:=b
\bigl(A(\alpha_n)+A_n\bigl(P_{n}+b_{\alpha_n}(I_E-P_{n}))\bigr)\in
H^{\infty}(\mathbb T, \mathcal B(E))
$$
and  by (\ref{lemma4.treghc1000001211}),
\begin{equation}\label{cvbnppp}
\Delta=b(A)=\Delta_n
\bigl(b_{\alpha_n}P_{n}+(I_E-P_{n})\bigr).
\end{equation}
Now write $B_n:=b_{\alpha_n}P_{n}+(I_E-P_{n})$. \ Then $(B_n
C)(\alpha_n)=(I_E-P_{n})C(\alpha_n)=0$. \ Thus we can write
$$
B_nC=(b_{\alpha_n}I_E)F \quad \hbox{for some} \ F \in
H^{\infty}(\mathbb D, \mathcal B(E)).
$$
Thus $P[\Theta] =AC=A B_n^* (b_{\alpha_n}I_E) F$ and hence, by
(\ref{lemma4.treghc1000001211}) and (\ref{cvbnppp}), we have  
$$
\Theta=\Delta_n (b_{\alpha_n}I_E)b(F),
$$
so that 
$$
\theta \overline{b_{\alpha_n}} I_E=\Delta_nb(F).
$$
This completes the proof.
\end{proof}

\medskip

We then have:

\begin{thm}\label{sxafcghjl} Let $\theta$ be a finite Blaschke product.
If $\Delta$ is an inner divisor of $\Theta=\theta I_E$, then
$\Delta$ is a finite Blaschke-Potapov product.
\end{thm}
\begin{proof} Suppose $\theta$ is a finite Blaschke product of the form
$$
\theta=\prod_{n=1}^M b_{\alpha_n}.
$$
If $\Delta$ is an inner divisor of $\Theta=\theta I_E$, then
$\Omega:=\Theta \Delta^*$ is also inner divisor of $\Theta$. \ Let
$P_n$ be the orthogonal projection of $E$ onto $\hbox{cl ran}\,
P[\Omega](\alpha_n)$. \ Then it follows from Lemma \ref{saxfvedrtyhb}
that $\Delta_1:=\Delta \bigl(b_{\alpha_M}P_{M}+(I_E-P_{M}) \bigr)^*$
is an inner divisor of $\theta \overline{b_{\alpha_M}} I_E$. \  By the
same argument we have that
$$
\aligned \Delta_2:&=\Delta_1
\bigl(b_{\alpha_{M-1}}P_{M-1}+(I_E-P_{M-1}) \bigr)^*\\
&=\Delta \bigl(b_{\alpha_M}P_{M}+(I_E-P_{M})
\bigr)^*\bigl(b_{\alpha_{M-1}}P_{M-1}+(I_E-P_{M-1}) \bigr)^*
\endaligned
$$
is an inner divisor of $\theta
\overline{b_{\alpha_{M-1}}}\overline{b_{\alpha_M}} I_E$. \ Continuing
this process, we have that
$$
\Delta_M:=\Delta
\prod_{n=0}^{M-1}\bigl(b_{\alpha_{M-n}}P_{M-n}+(I_E-P_{M-n})
\bigr)^*
$$
is an inner divisor of $I_E$. \ Thus $V\equiv \Delta_M$ is a unitary
operator, and hence
$$
\Delta=V \prod_{n=1}^{M}\bigl(b_{\alpha_{n}}P_{n}+(I_E-P_{n})
\bigr).
$$
is a finite Blaschke-Potapov product. \ This completes the proof.
\end{proof}

\medskip

\begin{cor}\label{sdcgbfujyp} A function
$\Phi \in H^{\infty}(\mathbb T, \mathcal B(D,
E))$ is rational if and only if
$$
\Phi=\Delta A^*,
$$
where $\Delta$ is a finite Blaschke-Potapov product and $A\in
H^\infty (\mathbb T, \mathcal B(E,D))$ is such that $\Delta$ and $A$
are right coprime.

\end{cor}

\begin{proof} This follows from Lemma \ref{rem2.4},
(\ref{rsftuy}) and Theorem \ref{sxafcghjl}.
\end{proof}

\medskip

We are ready for:

\medskip

\begin{cor} A two-sided inner function
$\Phi \in H^{\infty}(\mathbb T, \mathcal B(E))$ is rational if and
only if it can be represented as a finite Blaschke-Potapov product.
\end{cor}

\begin{proof} Suppose that $\Phi$ is rational and two-sided inner. \ Then it follows
from Corollary \ref{sdcgbfujyp} that
\begin{equation}\label{fghjhgvllll}
\Phi=\Delta A^*,
\end{equation}
where $\Delta$ is a finite Blaschke-Potapov product and $A\in
H^\infty (\mathbb T, \mathcal B(E))$. \ Since $\Phi$ and $\Delta$ are
two-sided inner, so is  $A$. \ Thus, by (\ref{fghjhgvllll}), $\Phi$
is a left inner divisor of $\Delta$, and hence the result follows
from Theorem \ref{sxafcghjl}. \ The converse is clear. \ This completes
the proof.
\end{proof}

\bigskip

%
%
%
%

\section{Coprime operator-valued rational functions}

\medskip

In this section we consider coprime operator-valued rational functions.

\begin{lem}\label{wsdfdvgnlypoi}
Let $\Phi \in H^{\infty}(\mathbb T, \mathcal B(E))$. \ If $
P[\Phi](\alpha)$ has no dense range for $\alpha\in\mathbb D$, then
$$
P:=b_{\alpha} P_M + (I-P_M) \quad (M:=\ker P[\Phi](\alpha)^*)
$$
is a nontrivial left inner divisor of $\Phi$.
\end{lem}

\begin{proof} Write $A:= P[\Phi]$; that is, $A$ is the Poisson integral of $\Phi$, defined by (\ref{Poisson}). \ Suppose that the range of $A(\alpha)$ is not dense. \ Then  $M:=\ker
A(\alpha)^*=\bigl(\hbox{cl ran}\, A(\alpha)\bigr)^{\perp} \neq
\{0\}$. \ Put $P:=b_{\alpha} P_M + (I-P_M)$. \ Then $(P^*b_{\alpha}
I_EA)(\alpha)=0$, and hence we can write
$$
P^*b_{\alpha} I_EA=b_{\alpha}I_E A_1 \quad \hbox{for some} \ A_1 \in
H^{\infty}(\mathbb D, \mathcal B(E)),
$$
which implies that $A=PA_1$. \ This completes the proof.
\end{proof}

\medskip
For an inner function $\theta$, let
$\mathcal{Z}(\theta)\label{ztheta}$ be the set of all zeros of
$\theta$. \ Then we have:

\begin{thm}\label{lem3.12}
Let $\Phi \in H^{\infty}(\mathbb T, \mathcal B(E))$ and
$\Theta:=\theta I_E$ with a finite Blaschke product $\theta$. \ Then
the following statements are equivalent:
\begin{itemize}
\item[(a)] $P[\Phi](\alpha)$ has dense range for each $\alpha \in
\mathcal{Z}(\theta)$;
\item[(b)] $\Phi$ and  $\Theta$ are left coprime.
\end{itemize}
\end{thm}

\begin{proof} (a) $\Rightarrow$ (b): Suppose that $\Phi$ and $\Theta$ are not left coprime. \ Then by
Theorem \ref{sxafcghjl}, there exist $\alpha_0 \in \mathcal
Z(\theta)$ and  a nonzero subspace $M$ of $E$ such that
$$
\Phi=\bigl(b_{\alpha_0}P_M+(I-P_M)\bigr) \Omega,
$$
where $\Omega \in H^{\infty}(\mathbb T, \mathcal B(E))$. \ Thus
$\hbox{cl ran}\, P[\Phi](\alpha_0) \subseteq M^{\perp} \neq E$.

\medskip

(b) $\Rightarrow$ (a): This follows from at once form Lemma
\ref{wsdfdvgnlypoi}.
\end{proof}

\medskip

\begin{lem}\label{corfgghh2.9}  If $\Phi\in H^{\infty}(\mathbb T, \mathcal B(D,E))$, then
$\widetilde{\Phi} \in H^{\infty}(\mathbb T, \mathcal B(E,D))$. \  In
this case,

$$
P[\widetilde{\Phi}]=\widetilde{P[\Phi]}
$$
\end{lem}
\begin{proof} Since $\widehat{\widetilde{\Phi}}(n)=\widehat{\Phi}(n)^*$ for all $n=0,1,2,\cdots$,  it follows that
$$
P[\widetilde{\Phi}](\zeta)=\sum_{n=0}^{\infty}\widehat{\Phi}(n)^*\zeta^n=\widetilde{P[\Phi]}(\zeta).
$$
\end{proof}

\begin{cor}\label{wswnlypoi} Let $\Phi \in H^{\infty}(\mathbb T, \mathcal B(E))$.
If $ P[\Phi](\alpha)$ is not injective for $\alpha\in\mathbb D$,
then
$$
P:=b_{{\alpha}} P_M + (I-P_M) \quad (M:=\ker P[\Phi](\alpha))
$$
is a  nontrivial right inner divisor of $\Phi$.
\end{cor}

\begin{proof} Suppose that $
P[\Phi](\alpha)$ is not injective. \ Then, by Lemma \ref{corfgghh2.9},
$(P[\Phi](\alpha))^*=P[\widetilde{\Phi}](\overline{\alpha})$ has no
dense range. \ Let
$$
Q:=b_{\overline{\alpha}} P_M + (I-P_M),
$$
where $M:=\ker P[\Phi](\alpha)=\ker
P[\widetilde{\Phi}](\overline{\alpha})^*\neq \{0\}$. \ Then it follow
from Lemma \ref{wsdfdvgnlypoi} that $Q$ is a nontrivial left inner
divisor of $\widetilde{\Phi}$. \ But since $Q$ is two-sided inner, it
follows that $P=\widetilde{Q}$ is a nontrivial right inner divisor
of $\Phi$. \ This completes the proof.
\end{proof}

\medskip

We also have:

\begin{cor}\label{edcqwerfwv12}
Let $\Phi \in H^{\infty}(\mathbb T, \mathcal B(E))$ and
$\Theta:=\theta I_E$ with a finite Blaschke product $\theta$. \ Then
the following statements are equivalent:
\begin{itemize}
\item[(a)] $P[\Phi](\alpha)$ is  injective for each $\alpha \in
\mathcal{Z}(\theta)$;
\item[(b)] $\Phi$ and  $\Theta$ are right coprime.
\end{itemize}
\end{cor}

\begin{proof}
Immediate from Theorem \ref{lem3.12} and Lemma \ref{corfgghh2.9}.
\end{proof}

\medskip

\begin{cor}\label{edfcdrfwv12}
Let $\Phi \in H^{\infty}(\mathbb T, M_n)$ and $\Theta:=\theta I_n$
with a finite Blaschke product $\theta$. \ Then the following
statements are equivalent:
\begin{itemize}
\item[(a)] $P[\Phi](\alpha)$ is invertible for each $\alpha \in
\mathcal{Z}(\theta)$;
\item[(b)] $\Phi$ and  $\Theta$ are right coprime;
\item[(c)] $\Phi$ and  $\Theta$ are left coprime.
\end{itemize}
\end{cor}

\begin{proof} The equivalence
(a) $\Leftrightarrow$ (b) follows from Theorem \ref{lem3.12} and
Corollary \ref{edcqwerfwv12} together with matrix theory. \ The
equivalence (b) $\Leftrightarrow$ (c) comes from \cite[Lemma
3.3]{CHL0}.
\end{proof}

\medskip

The equivalence (b) $\Leftrightarrow$ (c) of Corollary
\ref{edfcdrfwv12} may fail for operator-valued functions. \ For
example, if we take $E=\ell^2(\mathbb{Z}_+)$, then $S_E$ and $zI_E$
are right coprime, but not left coprime (cf. \cite{CHL3}).

\bigskip

%
%
%
%

\section{Miscellany}

\medskip

In this section, we establish some key differences between matrix-valued
functions and operator-valued functions.

\medskip


\subsection{A glance at right coprime-ness}

\medskip

If $\Phi$ and $\Psi$ are not left coprime then there exists a common
nontrivial left inner divisor $\Delta$ of both $\Phi$ and $\Psi$.
However we don't guarantee that this is still true for right
coprime-ness. \ In other words, if $\Phi$ and $\Psi$ are not right
coprime then by definition $\Phi=A\Delta$ and $\Psi=B\Delta$ for
some nontrivial inner function $\widetilde\Delta$. \ However we need
not expect that $\Delta$ is inner.

We here give a sufficient condition for the existence of a common
nontrivial right inner divisor of two functions when they are not
right coprime.

To see this, we first recall that a
 function $F\in H^{\infty}(\mathbb T, \mathcal B(E^{\prime}, E))$
is called {\it outer} if $\hbox{cl}\, F H^2(\mathbb T,
E^{\prime})=H^2(\mathbb T, E)$. \ We then have an analogue of the
scalar factorization theorem (called the {\it inner-outer
factorization}):

\bigskip\noindent
{\bf The inner-outer factorization} \cite{Ni1}.
If $A \in H^{\infty}(\mathbb T,
\mathcal B(D, E))$, then we can write
$$
A=A^iA^e \quad (\hbox{inner-outer factorization}),
$$
where $E^{\prime}$ is a subspace of $E$, $A^i\in H^{\infty}(\mathbb
T, \mathcal B (E^{\prime}, E))$ is an inner function, and $A^e \in
H^{\infty}(\mathbb T, \mathcal B(D, E^{\prime}))$ is an outer
function.

\medskip

The following lemma is a characterization of functions of bounded type.

\medskip

\begin{lem}\label{asgthyjskdfrg} \cite[Corollary 2.25.]{CHL3} Let
$\Omega$ be an inner function with values in $\mathcal B(D, E)$.
Then
$$
\breve{\Omega} \ \hbox{is of bounded type} \Longleftrightarrow
[\Omega, \Omega_c] \ \hbox{is two-sided inner},
$$
where $\Omega_c$ is the complementary factor of $\Omega$, i.e.,
$\hbox{ker}\,H_{\Omega^*} =[\Omega, \Omega_c]\, H^2_{D\oplus
D^\prime}$ for some Hilbert space $D^\prime$, and $[\Omega, \Omega_c]$ denotes the $1 \times 2$ matrix whose entries are $\Omega$ and $\Omega_c$.
\end{lem}

\medskip

We then have:

\begin{thm}\label{regh}
Suppose that $\Phi\in H^{\infty}(\mathbb T, \mathcal B(D,E_1))$ and $\Psi\in
H^{\infty} (\mathbb T, \mathcal B(D,E_2))$ are not right coprime. \ If
there exists a nontrivial left inner divisor $\Omega$ of
$\Delta:=\hbox{left-g.c.d.}(\widetilde{\Phi}, \widetilde{\Psi})$ and
$\breve{\Omega}$ is of bounded type, then $\widetilde{[\Omega,
\Omega_c]}$ is a common nontrivial right inner divisor of both
$\Phi$ and $\Psi$.
\end{thm}

\begin{proof} Since $\Phi\in H^{\infty}(\mathbb T, \mathcal B(D,E_1))$ and $\Psi\in
H^{\infty} (\mathbb T, \mathcal B(D,E_2))$ are not right coprime,
$\Delta:=\hbox{left-g.c.d.}(\widetilde{\Phi}, \widetilde{\Psi}) \in
H^{\infty}(\mathbb T, \mathcal B(D_1, D))$ is not a unitary operator
and we can write
\begin{equation}\label{sxadfghjk}
\widetilde{\Phi}=\Delta \widetilde{\Phi}_1 \quad \hbox{and} \quad
\widetilde{\Psi}=\Delta \widetilde{\Psi}_1 ,
\end{equation}
where $\widetilde{\Phi}_1 \in H^{\infty}(\mathbb T, \mathcal B(E_1,
D_1))$ and $\widetilde{\Psi}_1 \in H^{\infty}(\mathbb T, \mathcal
B(E_2, D_1))$. \ Since $\Omega$ is a left inner divisor of $\Delta$, we
can write
\begin{equation}\label{sdfglhtkjhg}
\Delta=\Omega \Delta_1 \quad(\Delta_1 \in H^{\infty}(\mathbb T,
\mathcal B(D_1, D_2)) \quad(\Omega \in
H^{\infty}(\mathbb T, \mathcal B(D_2, D))).
\end{equation}
Since $\breve{\Omega}$ is not a unitary operator and is of bounded
type, by Lemma \ref{asgthyjskdfrg}, $\Omega_0\equiv[\Omega,
\Omega_c]$ is not a unitary operator and  a two-sided inner function.
Note that
$\Omega_c$ is an inner function with values in $\mathcal B(D_3, D)$ for
some Hilbert space $D_3$. \ Thus we can write
$$
\Delta=\Omega \Delta_1=[\Omega,
\Omega_c]\begin{bmatrix}\Delta_1\\{\mathbf 0}\end{bmatrix} \equiv
\Omega_0\Delta_0\quad (\hbox{where}\ {\mathbf 0}:D_4 \to D_3).
$$
It thus follows from (\ref{sxadfghjk}) and(\ref{sdfglhtkjhg}) that
$$
\Phi=\Phi_1\widetilde{\Delta}_0\widetilde{\Omega}_0 \quad \hbox{and}
\quad  \Psi=\Psi_1\widetilde{\Delta}_0\widetilde{\Omega}_0.
$$
But since $\Omega_0$ is two-sided inner, we have that
$\widetilde{\Omega}_0$ is (two-sided) inner. \ This completes the
proof.
\end{proof}

\medskip

\begin{cor} Let $\Phi\in H^{\infty}(\mathbb T, \mathcal
B(D,E_1))$, $\Psi\in H^{\infty} (\mathbb T, \mathcal B(D,E_2))$ and
$\Delta:=\hbox{left-g.c.d.}(\widetilde{\Phi}, \widetilde{\Psi})$.
 If $\breve{\Delta}$ is of bounded type, then $\widetilde{[\Delta,
\Delta_c]}$ is a common nontrivial right inner divisor of both
$\Phi$ and $\Psi$.
\end{cor}

\begin{proof} Immediate from Theorem \ref{regh}.
\end{proof}

\bigskip


\subsection{Subnormality of Toeplitz operators}

\medskip

In 1970, P.R. Halmos posed the following problem, listed as Problem
5 in his series of lectures, ``Ten problems in Hilbert space" \cite{Ha1}:
$$
\hbox{Is every subnormal Toeplitz operator either normal or
analytic\,?}
$$
Halmos' Problem 5 has been partially answered in the affirmative by
many authors. \ However, in 1984, Halmos' Problem 5 was answered in
the negative by C. Cowen and J. Long \cite{CoL}. \ Despite considerable efforts, to date researchers have been unable to characterize subnormal Toeplitz
operators in terms of their symbols. \ Thus we have:

\medskip

\noindent {\bf Halmos' Problem 5 reformulated.} {\it Which Toeplitz
operators are subnormal\,?}

\medskip

For cases of matrix-valued symbols, the subnormality of Toeplitz
operators was studied in \cite{CHL0}, in which it was shown that if
the matrix-valued symbol $\Phi$ satisfies a general condition on
coprime factorization and $T_\Phi$ is subnormal then it is either
normal or analytic. \ Also in \cite{CHKL}, it was conjectured that
every subnormal Toeplitz operator with matrix-valued rational symbol
is unitarily equivalent to a direct sum of a normal operator and a
Toeplitz operator with analytic symbol. \ In fact, if an $n \times n$
matrix-valued function $\Phi$ is analytic then the normal extension
of $T_\Phi$ is the multiplication operator $M_{\Phi}$, so clearly
$T_\Phi$ is subnormal. \ However, this is not the case for the
operator-valued symbols. \ In this section we will give an example
(see Example \ref{ex57} below). \ On the other hand, if $\Phi$ is
matrix-valued and $T_\Phi$ is subnormal (even hyponormal), then
$\Phi$ should be normal, i.e., $\Phi^*\Phi=\Phi\Phi^*$ a.e. on
$\mathbb T$ (cf. \cite{GHR}). \ However this may also fail for
operator-valued symbols.

\medskip

\begin{ex} Let $S:=T_z$  on $H^2(\mathbb T)$ and
$\Phi(z)=Sz^n \in H^{\infty}(\mathbb T, \mathcal B(H^2(\mathbb T)))$
($n \geq 0$). \ Then
$$
T_{\Phi}^*T_{\Phi}=T_{S^*S}=I_{H^2(\mathbb T, H^2(\mathbb T))},
$$
so that $T_{\Phi}$ is quasinormal and hence subnormal. \ However,
$$
\Phi(z)\Phi^*(z)=SS^* \neq S^*S=\Phi^*(z)\Phi(z) \quad \hbox{for
all} \ z \in \mathbb T,
$$
which implies that $\Phi$ is not normal. \ Here we don't need to
expect that the multiplication operator $M_{\Phi}:L^2(\mathbb T,
\mathcal B(H^2(\mathbb T))) \rightarrow L^2(\mathbb T, \mathcal
B(H^2(\mathbb T)))$ is a normal extension of $T_{\Phi}$. \ Indeed, it
is easy to show that $M_{\Phi}$ is not normal, and hence $M_\Phi$
can never be a normal extension of $T_{\Phi}$. \ What is a normal
extension of $T_\Phi$\,? Let $B:=M_z$ on $L^2(\mathbb T)$ and
$\Psi(z):=Bz^n \in H^{\infty}(\mathbb T, \mathcal B(L^2(\mathbb
T)))$. \ Then a straightforward calculation shows that the
multiplication operator $ M_{\Psi}: L^2(\mathbb T, \mathcal
B(L^2(\mathbb T)))\rightarrow L^2(\mathbb T, \mathcal B(L^2(\mathbb
T)))$ is a normal extension of $T_{\Phi}$.
\end{ex}

\medskip

The following simple example shows that analytic Toeplitz operators with operator-valued symbols 
need not be even hyponormal.

\begin{ex}\label{ex57}
 Let $\Phi(z)=S^* \in H^{\infty}(\mathbb T, \mathcal B(H^2(\mathbb
T)))$ and $e_0$ be the costant function ${\mathbf 1} \in H^2(\mathbb
T)$. \ If $f(z)=e_0z$, then
$$
\langle (T_{\Phi^*}T_{\Phi}-T_{\Phi}T_{\Phi^*})f, \ f
\rangle=\langle -e_0z, e_0z\rangle=-1<0,
$$
which implies that  $T_{\Phi}$  is not hyponormal and hence not
subnormal even though $\Phi$ is analytic.
\end{ex}

We would like to pose:

\begin{q} Which analytic Toeplitz operators with operator-valued symbols
are subnormal ?
\end{q}

For a sufficient condition, one may be tempted to conjecture that if $\Phi \in
H^{\infty}(\mathbb T, \mathcal B(H^2(\mathbb T)))$ and if $\Phi(z)$
is subnormal for almost all $z\in\mathbb T$, then $T_{\Phi}$ is
subnormal. \ We have not been able to decide whether this is true.


\bigskip

\noindent {\it Acknowledgment}. \ The work of the second named
author was supported by NRF(Korea) grant No. 2019R1A2C1005182. \  The
work of the third named author was supported by NRF(Korea) grant No.
2021R1A2C1005428.

\bigskip

%
%

\vskip 1cm

\vskip 1cm

Ra{\'u}l\ E.\ Curto

Department of Mathematics, University of Iowa, Iowa City, IA 52242,
U.S.A.

E-mail: raul-curto@uiowa.edu

\bigskip

In Sung Hwang

Department of Mathematics, Sungkyunkwan University, Suwon 16419,
Korea

E-mail: ihwang@skku.edu

\bigskip

Woo Young Lee

Department of Mathematics and RIM, Seoul National University, Seoul
08826, Korea

E-mail: wylee@snu.ac.kr


\end{document}